\documentclass[12pt]{amsart}
 \usepackage{geometry}                
\usepackage{graphicx}
\usepackage{amssymb}
\usepackage{epstopdf}
\usepackage{color}
\usepackage{bbm, dsfont}
\usepackage{hyperref}
\usepackage[utf8]{inputenc}
\usepackage{amssymb,amsthm,amsmath,amssymb,wrapfig,dsfont}
\usepackage[dvipsnames]{xcolor}
\definecolor{myred}{RGB}{251,154,133}
\definecolor{myblue}{RGB}{153,206,227}
\definecolor{mylightblue}{RGB}{0, 150, 255}
\definecolor{mygreen}{RGB}{32, 210, 64}
\definecolor{mygray}{RGB}{220, 220, 220}

\usepackage{tikz}
\usetikzlibrary{decorations.pathmorphing}
\tikzset{snake it/.style={decorate, decoration=snake}}
\usetikzlibrary{shapes.geometric,positioning,decorations.pathreplacing}

\newtheorem{theorem}{Theorem}[section]
\newtheorem{lemma}[theorem]{Lemma}

\newtheorem{proposition}[theorem]{Proposition}
\newtheorem{conjecture}[theorem]{Conjecture}
\theoremstyle{definition}

\theoremstyle{remark}
\newtheorem{remark}[theorem]{Remark}

\theoremstyle{example}

\DeclareFontFamily{OML}{rsfs}{\skewchar\font'177}
\DeclareFontShape{OML}{rsfs}{m}{n}{ <5> <6> rsfs5 <7> <8> <9>
rsfs7 <10> <10.95> <12> <14.4> <17.28> <20.74> <24.88> rsfs10 }{}
\DeclareMathAlphabet{\mathfs}{OML}{rsfs}{m}{n}

\newcommand{\BZ}{{\mathbb{Z}}}

\newcommand{\bae}{\begin{equation}\begin{aligned}}
\newcommand{\eae}{\end{aligned}\end{equation}}

\newcommand{\ep}{{\varepsilon}}

\def\beq{ \begin{equation} }
\def\eeq{ \end{equation} }
\def\mn{\medskip\noindent}

\def\square{\vcenter{\vbox{\hrule height .4pt
  \hbox{\vrule width .4pt height 5pt \kern 5pt
        \vrule width .4pt} \hrule height .4pt}}}

\def\ZZ{\mathbb{Z}}

\begin{document}

\numberwithin{equation}{section} 

\title{On Covering paths with 3 Dimensional Random Walk}

\author{Eviatar B. Procaccia}
\address[Eviatar B. Procaccia\footnote{Research supported by NSF grant DMS-1407558}]{Texas A\&M University}
\urladdr{www.math.tamu.edu/~procaccia}
\email{eviatarp@gmail.com}
 
\author{Yuan Zhang}
\address[Yuan Zhang]{Texas A\&M University}
\urladdr{http://www.math.tamu.edu/~yzhang1988/}
\email{yzhang1988@math.tamu.edu}

\maketitle


\begin{abstract}
In this paper we find an upper bound for the probability that a $3$ dimensional simple random walk covers each point in a nearest neighbor path connecting 0 and the boundary of an $L_1$ ball of radius $N$. For $d\ge 4$, it has been shown in \cite{cover_probability_d} that such probability decays exponentially with respect to $N$. For $d=3$, however, the same technique does not apply, and in this paper we obtain a slightly weaker upper bound: $\forall \ep>0,\exists c_\ep>0,$  $$P\left({\rm Trace}(\mathcal{P})\subseteq {\rm Trace}\big(\{X_n\}_{n=0}^\infty\big) \right)\le \exp\left(-c_\ep N\log^{-(1+\ep)}(N)\right).$$ 
\end{abstract}
\section{Introduction}
In this paper, we study the probability that the trace of a nearest neighbor path in $\ZZ^3$ connecting 0 and the boundary of a $L_1$ ball of radius $N$ is completely covered by the trace of a  $3$ dimensional simple random walk.

First, we review some results we proved in a recent paper for general $d$'s. For any integer $N\ge 1$, let $\partial B_1(0,N)$ be the boundary of the $L_1$ ball in $\ZZ^d$ with radius $N$. We say that a nearest neighbor path
$$
\mathcal{P}=\big(P_0,P_1,\cdots, P_K\big)
$$
is connecting 0 and $\partial B_1(0,N)$ if $P_0=0$ and $\inf\{n: \|P_n\|_1= N\}=K$. And we say that a path $\mathcal{P}$ is covered by a $d$ dimensional random walk $\{X_{d,n}\}_{n=0}^\infty$ if
$$
{\rm Trace}(\mathcal{P})\subseteq {\rm Trace}(X_{d,0},X_{d,1},\cdots):=\{x\in\BZ^d,\exists n~X_{d,n}=x\}.
$$
In \cite{cover_probability_d}, we have shown that for any $d\ge 2$ such covering probability is maximized over all nearest neighbor paths connecting 0 and $\partial B_1(0,N)$ by the monotonic path that stays within distance one above/below the diagonal $x_1=x_2=\cdots=x_d$. 
\begin{theorem}(Theorem 1.4 in \cite{cover_probability_d})
\label{theorem D random walk}
For each integers $L\ge N\ge 1$, let $\mathcal{P}$ be any nearest neighbor path in $\ZZ^d$ connecting 0 and $\partial B_1(0,N)$. Then
$$
P\big({\rm Trace}(\mathcal{P})\in {\rm Trace}(X_{d,0},\cdots, X_{d,L})\big)\le P\big(\overset{\nearrow}{\mathcal{P}}\in {\rm Trace}(X_{d,0},\cdots, X_{d,L})\big)
$$
where 
$$
\overset{\nearrow}{\mathcal{P}}=\Big({\rm arc}_1[0:d-1],{\rm arc}_2[0:d-1],\cdots, {\rm arc}_{[N/d]}[0:d-1], {\rm arc}_{[N/d]+1}[0:N-d[N/d]] \Big),
$$
$$
{\rm arc}_1[0:d-1]=\left(0, e_1, e_1+e_2,\cdots, \sum_{i=1}^{d-1}e_i\right) 
$$
and ${\rm arc}_k=(k-1)\sum_{i=1}^{d}e_i+{\rm arc}_1$. 
\end{theorem} 
Then noting that the probability of covering $\overset{\nearrow}{\mathcal{P}}$ is bounded above by the probability a simple random walk returns to the exact diagonal line for $[N/d]$ times, one can introduce the Markov process 
$$
\hat X_{d-1,n}=\left(X_{d,n}^1-X_{d,n}^2,X_{d,n}^2-X_{d,n}^3,\cdots, X_{d,n}^{d-1}-X_{d,n}^{d}\right)
$$
where $X_{d,n}^i$ is the $i$th coordinate of $X_{d,n}$ and see that $\{\hat X_{d-1,n}\}_{n=0}^\infty$ is another $d-1$ dimensional non simple random walk, which is transient when $d\ge 4$. Thus, we immediately have the following upper bound: 
\begin{theorem}(Theorem 1.5 in \cite{cover_probability_d})
\label{Theorem 1.5 d}
There is a $P_d\in (0,1)$ such that for any nearest neighbor path $\mathcal{P}=(P_0,P_1,\cdots, P_K)$ connecting 0 and $\partial B_1(0,N)$ and $\{X_{d,n}\}_{n=0}^\infty$ which is a $d-$dimensional simple random walk starting at 0 with $d\ge 4$, we always have 
$$
P\left({\rm Trace}(\mathcal{P})\subseteq {\rm Trace}\big(\{X_{d,n}\}_{n=0}^\infty\big) \right)\le P_d^{[N/d]}.  
$$
Here $P_d$ equals to the probability that $\{X_{d,n}\}_{n=0}^\infty$ ever returns to the $d$ dimensional diagonal line.  
\end{theorem}
Theorem \ref{Theorem 1.5 d} implies that for each fixed $d\ge 4$, the covering probability decays exponentially with respect to $N$. 

For $d=3$, the same technique we had may not hold since now $\{\hat X_{2,n}\}_{n=0}^\infty$ is a recurrent 2 dimensional random walk, which means that $P_d=1$ and that the original 3 dimensional random walk will return to the diagonal line infinitely often. To overcome this issue, we note that although the diagonal line   
$$
\mathcal{D}_\infty=\{(0,0,0),(1,1,1),\cdots\}
$$
is recurrent, it is possible to find an infinite subset $\mathcal{\tilde D}_\infty$ that is transient. And if we can further show for this specific transient subset we found, the returning probability is uniformly bounded away from 1 (which is not generally true for all transient subsets, as is shown in Counterexample 1 in Section 3), then we are able to show 
$$
P\big(\overset{\nearrow}{\mathcal{P}}\in {\rm Trace}(X_{3,0},X_{3,1},\cdots)\big)\le \exp\left(-c\Big|\mathcal{\tilde D}_\infty \cap \overset{\nearrow}{\mathcal{P}} \Big|\right).
$$
With this approach, we have the following theorem 
\begin{theorem}
\label{Theorem 3}
For each $\ep>0$, there is a $c_\ep \in (0,\infty)$ such that for any $N\ge 2$ and any nearest neighbor path $\mathcal{P}=(P_0,P_1,\cdots, P_K)\subset \ZZ^3$ connecting 0 and $\partial B_1(0,N)$, we have 
$$ 
P\left({\rm Trace}(\mathcal{P})\subseteq {\rm Trace}\big(\{X_{3,n}\}_{n=0}^\infty\big) \right)\le \exp\left(-c_\ep N\log^{-(1+\ep)}(N)\right). 
$$
\end{theorem} 

Note that the upper bound in Theorem \ref{Theorem 3} seems to be non-sharp. The reason is that we did not fully use the geometric property of path $\overset{\nearrow}{\mathcal{P}}$ to minimize the covering probability. I.e., although we require our simple random walk to visit the transient subset for $O(N\log^{-1-\ep}(N))$ times, those returns may be not enough to cover  every point in $\mathcal{\tilde D}_\infty \cap \overset{\nearrow}{\mathcal{P}}$. We conjecture that the actual decay rate is also exponential for $d=3$. Numerical simulations seem to support this as is shown in Section 5. 

\begin{conjecture}
\label{conjecture 3}
There is a $c \in (0,\infty)$ such that for any $N\ge 2$ and any nearest neighbor path $\mathcal{P}=(P_0,P_1,\cdots, P_K)\subset \ZZ^3$ connecting 0 and $\partial B_1(0,N)$, we always have 
$$ 
P\left({\rm Trace}(\mathcal{P})\subseteq {\rm Trace}\big(\{X_{3,n}\}_{n=0}^\infty\big) \right)\le \exp\left(-c N\right). 
$$
\end{conjecture}

The structure of this paper is as follows: In Section 2, we construct the infinite subset $\mathcal{\tilde D}_\infty$ of the diagonal line, calculate its density, and show it is transient. In Section 3, we show the returning probability of $\mathcal{\tilde D}_\infty$ is uniformly bounded away from 1, no matter where on $\mathcal{\tilde D}_\infty$ the random walk starts from. With these results, in Section 4, we finish the proof of Theorem \ref{Theorem 3}. Numerical simulations are given in Section 5 showing possible non tightness of our result.

\section{Infinite Transient Subset of the Diagonal}
Without loss of generality we can concentrate on the proof of Theorem \ref{Theorem 3} for sufficiently large $N$. Recall that 
 $$
\overset{\nearrow}{\mathcal{P}}=\Big({\rm arc}_1[0:d-1],{\rm arc}_2[0:d-1],\cdots, {\rm arc}_{[N/d]}[0:d-1], {\rm arc}_{[N/d]+1}[0:N-d[N/d]] \Big)
$$
is the path connection 0 and $B_1(0,N)$ that maximizes the covering probability. When $d=3$, let 
$$
\mathcal{D}_{[N/3]}=\{(0,0,0),(1,1,1),\cdots, ([N/3],[N/3],[N/3])\}
$$
be the points in $\overset{\nearrow}{\mathcal{P}}$ that lie exactly on the diagonal. Although it is clear that for simple random walk $\{X_{3,n}\}_{n=0}^\infty$ starting at 0,  $\mathcal{D}_\infty$ is a recurrent set, following a similar construction to Spitzer \cite[Chapter 6.26]{spitzer}, we find a transient infinite subset of $\mathcal{D}_\infty$ as follows: for $n_1=0$, $n_2=\lceil \log^{1+\ep}(2)\rceil=1$, and for all $k\ge 3$
\beq
\label{equation step length}
n_k=\left\lceil \sum_{i=1}^k\log^{1+\ep}(i)\right\rceil\in \ZZ,
\eeq
define
$$
\mathcal{\tilde D}_\infty=\{(n_k,n_k,n_k)\}_{k=1}^\infty\subset \mathcal{D}_\infty. 
$$
Since $\log^{1+\ep}(k)>1$ for all $k\ge 3$, it is easy to see that $\{n_k\}_{k=1}^\infty$ is a monotonically increasing sequence. Moreover, for each $1\le k_1<k_2<\infty$, 
\begin{align*}
n_{k_2}-n_{k_1}&=\left\lceil \sum_{i=1}^{k_2}\log^{1+\ep}(i)\right\rceil-\left\lceil \sum_{i=1}^{k_1}\log^{1+\ep}(i)\right\rceil\\
&\ge \sum_{i=k_1+1}^{k_2}\log^{1+\ep}(i)-1.
\end{align*}
This implies that for all $k_2\ge 8$ and $1\le k_1<k_2$, 
\beq
\label{lower bound difference n_k}
n_{k_2}-n_{k_1} \ge \frac{1}{2}\int_{k_1}^{k_2} \log^{1+\ep}(x)dx.
\eeq

For any $N\in\ZZ$, define 
$$
\mathcal{\tilde D}_N=\mathcal{\tilde D}_\infty\cap \mathcal{D}_{N}
$$
and
$$
C_N=\left|\mathcal{\tilde D}_N\right|=\sup\{k: \  n_k\le N\}. 
$$
Recalling the definition of $n_k$ in \eqref{equation step length}, we also equivalently have 
$$
C_N=\sup\left\{k: \sum_{i=1}^k\log^{1+\ep}(i)\le N \right\}=\inf\left\{k: \sum_{i=1}^k\log^{1+\ep}(i)> N \right\}-1. 
$$
\begin{lemma}
\label{lemma C_N}
For any $\ep>0$, there is a constant $C_\ep<\infty$ such that 
$$
C_N\in \left(2^{-1-\ep}N\log^{-1-\ep}(N), C_\ep N\log^{-1-\ep}(N)\right)
$$
for all $N\ge 2$. 
\end{lemma}
\begin{proof}
Noting that for any $k$ such that 
$$
\sum_{i=1}^k\log^{1+\ep}(i)> N
$$
we must have that $k>C_N$, and that 
\beq
\label{C_N 1}
\begin{aligned}
\sum_{i=1}^k\log^{1+\ep}(i)\ge \int_1^{k} \log^{1+\ep}(x)dx\ge \frac{1}{2^{1+\ep}}\big(k-k^{1/2}\big) \log^{1+\ep}(k),
\end{aligned}
\eeq
for $K_N=\left\lceil 2^{2+\ep}N/\log^{1+\ep}(N)\right\rceil$, we have by \eqref{C_N 1}
\beq
\label{C_N 2}
\begin{aligned}
\sum_{i=1}^{K_N}\log^{1+\ep}(i)&\ge \frac{1}{2^{1+\ep}}\big(K_N-K_N^{1/2}\big) \log^{1+\ep}(K_N)\\
&\ge   \frac{1}{2^{1+\ep}} \cdot K_N\cdot \frac{K_N-K_N^{1/2}}{K_N}\cdot \log^{1+\ep}\left(2^{2+\ep}N/\log^{1+\ep}(N)\right)\\
&\ge  2N\cdot  \frac{K_N-K_N^{1/2}}{K_N}\cdot\frac{\log^{1+\ep}\left(2^{2+\ep}N/\log^{1+\ep}(N)\right)}{\log^{1+\ep}(N)}. 
\end{aligned}
\eeq
Noting that $K_N\to\infty$ as $N\to\infty$ and that 
$$
\lim_{N\to\infty}\frac{\log^{1+\ep}\left(\log^{1+\ep}(N)\right)}{\log^{1+\ep}(N)}=\lim_{N\to\infty} (1+\ep)^{1+\ep} \left[\frac{\log(\log(N))}{\log(N)}\right]^{1+\ep}=0, 
$$
for sufficiently large $N$
\beq
\sum_{i=1}^{K_N}\log^{1+\ep}(i)\ge2N\cdot  \frac{K_N-K_N^{1/2}}{K_N}\cdot\frac{\log^{1+\ep}\left(2^{2+\ep}N/\log^{1+\ep}(N)\right)}{\log^{1+\ep}(N)}>N
\eeq
which implies $C_N< K_N$ and finishes the proof of the upper bound. On the other hand, note that 
$$
\sum_{i=1}^k\log^{1+\ep}(i)\le \int_1^{k+1} \log^{1+\ep}(x)dx\le k \log^{1+\ep}(k+1).
$$
So for any $k\le 2^{-1-\ep}N\log^{-1-\ep}(N)$,
$$
\sum_{i=1}^k\log^{1+\ep}(i)\le k \log^{1+\ep}(k+1)\le 2^{-1-\ep}N\frac{\log^{1+\ep}\left(2^{-1-\ep} N\log^{-1-\ep}(N)+1 \right)}{\log^{1+\ep}(N)}<N. 
$$
Thus the proof of Lemma \ref{lemma C_N}  is complete. 
\end{proof}
With Lemma \ref{lemma C_N}, we next show that $\mathcal{\tilde D}_\infty$ is transient for 3 dimensional simple random walk: 
\begin{lemma}
\label{lemma transient}
For 3 dimensional simple random walk $\{X_{3,n}\}_{n=0}^\infty$, $\mathcal{\tilde D}_\infty$ is a transient subset. 
\end{lemma}
\begin{proof}
According to Wiener's test (see Corollary 6.5.9 of \cite{randomwalkbook}), it is sufficient to show that 
\beq
\sum_{k=1}^\infty 2^{-k} {\rm cap}(A_k)<\infty 
\eeq
where $A_k=\mathcal{\tilde D}_\infty\cap\left[B_2(0,2^{k})\setminus B_2(0,2^{k-1})\right]$. Then according to the definition of capacity (see Section 6.5 of  \cite{randomwalkbook}), we have for all $k\ge 1$
\beq
\label{equation transient 1}
{\rm cap}(A_k)\le |A_k|\le \left|\mathcal{\tilde D}_\infty\cap B_2(0,2^{k})\right|\le \left|\mathcal{\tilde D}_{2^k}\right|=C_{2^k}. 
\eeq
By Lemma \ref{lemma C_N},
\beq
\label{equation transient 2}
{\rm cap}(A_k)\le C_{2^k}\le \frac{C_\ep}{\log^{1+\ep}(2)} \frac{2^k}{k^{1+\ep}}. 
\eeq
Thus we have 
$$
\sum_{k=1}^\infty 2^{-k} {\rm cap}(A_k)\le \frac{C_\ep}{\log^{1+\ep}(2)}  \sum_{k=1}^\infty  \frac{1}{k^{1+\ep}}<\infty
$$
which implies that $\mathcal{\tilde D}_\infty$ is transient. 
\end{proof}

\section{Uniform Upper Bound on Returning Probability}
Now we have $\mathcal{\tilde D}_\infty$ is transient, i.e.,
$$
P\left(X_n\in \mathcal{\tilde D}_\infty \  \text{i.o.}\right)=0,
$$
which immediately implies that there must be some $\bar x\in \ZZ^3\setminus \mathcal{\tilde D}_\infty$ such that 
\beq
\label{visiting probability<1}
P_{\bar x}(T_{\mathcal{\tilde D}_\infty}<\infty)<1,
\eeq
where $T_{\mathcal{\tilde D}_\infty}$ is the first time a simple random walk visits $\mathcal{\tilde D}_\infty$, and $P_x(\cdot)$ is the distribution of the simple random walk condition on it starting at $x$. Then note that $\mathcal{\tilde D}_\infty$ is a subset of the diagonal line, which implies $\mathcal{\tilde D}_\infty$ has no interior point while $\ZZ^3\setminus \mathcal{\tilde D}_\infty$ is connected. Thus for any $x_k\in \mathcal{\tilde D}_\infty$, there exists a nearest neighbor path 
$$
\mathcal{Y}=\{y_0,y_1,\cdots y_m\}
$$
with $y_0=x_k$, $y_m=\bar x$ while $y_i\in \ZZ^3\setminus \mathcal{\tilde D}_\infty$, for all $i=1,2,\cdots, m-1$. Combining this with the fact that 
$$
P_x(T_{\mathcal{\tilde D}_\infty}<\infty)=\frac{1}{6}\sum_{i=1}^3 \left[P_{x+e_i}(T_{\mathcal{\tilde D}_\infty}<\infty)+P_{x-e_i}(T_{\mathcal{\tilde D}_\infty}<\infty)\right]
$$
for all $x\in \ZZ^3\setminus \mathcal{\tilde D}_\infty$, we have
$$
P_{y_i}(T_{\mathcal{\tilde D}_\infty}<\infty)<1, 
$$
for all $i\ge 1$, which in turns implies that 
\beq
\label{returning probability<1}
P_{x_k}(\bar T_{\mathcal{\tilde D}_\infty}<\infty)<1
\eeq
for all $k$, where $\bar T_{\mathcal{\tilde D}_\infty}$ is the first returning time, i.e. the stopping time a simple random walk first visit $\mathcal{\tilde D}_\infty$ after its first step. 

However,  in order to use the transient set $\mathcal{\tilde D}_\infty$ as if it is just like one point in a transient random walk, \eqref{returning probability<1} is not enough. We need to show that starting from each point $x_k=(n_k,n_k,n_k)\in\mathcal{\tilde D}_\infty$, the probability $P_{x_k}(\bar T_{\mathcal{\tilde D}_\infty}<\infty)$ is uniformly bounded away from 1. And this is not generally true for all transient subsets $A$. First of all, when $A$ has interior points, the returning probability of those points are certainly one. And even if  $A$ has no interior point and $\ZZ^3\setminus A$ is connected, we have the following counter example:

\mn {\bf Counterexample 1:} Consider subsets 
$$
A_k=\{(2^k,1,n),(2^k,-1,n),(2^k+1,0,n),(2^k-1,0,n)\}_{n=-k}^k\cup\{(2^k,0,0)\}
$$
and 
$$
A=\bigcup_{k=1}^\infty A_k
$$
where the 2 dimensional projection of $A$ is illustrated in Figure \ref{counterexample} (the distances between $A_k$'s are not exact in the figure):

\begin{figure}[h!]
\centering
\begin{tikzpicture}
\tikzstyle{redcirc}=[circle,
draw=black,fill=myred,thin,inner sep=0pt,minimum size=1.5mm]
\tikzstyle{bluecirc}=[circle,
draw=black,fill=myblue,thin,inner sep=0pt,minimum size=1.5mm]

\node (v1) at (0,0) [bluecirc] {};
\node (v2) at (0.25,0) [bluecirc] {};
\node (v3) at (0.5,0) [bluecirc] {};
\node (v4) at (0.5,0.25) [bluecirc] {};
\node (v5) at (0,0.25) [bluecirc]{};
\node (v6) at (0.5,-0.25) [bluecirc] {};
\node (v7) at (0,-0.25) [bluecirc]{};

\draw [thick] (v1) to (v2);
\draw [thick] (v1) to (v5);
\draw [thick] (v2) to (v3);
\draw [thick] (v3) to (v4);
\draw [thick] (v3) to (v6);
\draw [thick] (v1) to (v7);


\node (v10) at (2,0) [bluecirc] {};
\node (v20) at (2.25,0) [bluecirc] {};
\node (v25) at (2.5,0.5) [bluecirc] {};
\node (v30) at (2.5,0) [bluecirc] {};
\node (v40) at (2.5,0.25) [bluecirc] {};
\node (v50) at (2,0.25) [bluecirc]{};
\node (v55) at (2,0.5) [bluecirc]{};
\node (v60) at (2.5,-0.25) [bluecirc] {};
\node (v65) at (2.5,-0.5) [bluecirc] {};
\node (v70) at (2,-0.25) [bluecirc]{};
\node (v75) at (2,-0.5) [bluecirc]{};

\draw [thick] (v10) to (v20);
\draw [thick] (v10) to (v50);
\draw [thick] (v20) to (v30);
\draw [thick] (v30) to (v40);
\draw [thick] (v50) to (v55);
\draw [thick] (v40) to (v25);
\draw [thick] (v10) to (v70);
\draw [thick] (v70) to (v75);
\draw [thick] (v30) to (v60);
\draw [thick] (v60) to (v65);


\node (v100) at (6,0) [bluecirc] {};
\node (v200) at (6.25,0) [bluecirc] {};
\node (v250) at (6.5,0.5) [bluecirc] {};
\node (v300) at (6.5,0) [bluecirc] {};
\node (v400) at (6.5,0.25) [bluecirc] {};
\node (v500) at (6,0.25) [bluecirc]{};
\node (v550) at (6,0.5) [bluecirc]{};

\node (v101) at (6,0.75) [bluecirc]{};
\node (v102) at (6.5,0.75) [bluecirc]{};

\node (v103) at (6,-0.25) [bluecirc] {};
\node (v104) at (6,-0.5) [bluecirc] {};
\node (v105) at (6,-0.75) [bluecirc] {};

\node (v303) at (6.5,-0.25) [bluecirc] {};
\node (v304) at (6.5,-0.5) [bluecirc] {};
\node (v305) at (6.5,-0.75) [bluecirc] {};

\draw [thick] (v100) to (v200);
\draw [thick] (v100) to (v500);
\draw [thick] (v200) to (v300);
\draw [thick] (v300) to (v400);
\draw [thick] (v500) to (v550);
\draw [thick] (v400) to (v250);

\draw [thick] (v550) to (v101);
\draw [thick] (v250) to (v102);

\draw [thick] (v100) to (v103);
\draw [thick] (v103) to (v104);
\draw [thick] (v104) to (v105);

\draw [thick] (v300) to (v303);
\draw [thick] (v303) to (v304);
\draw [thick] (v304) to (v305);


\node (v1000) at (13,0) [bluecirc] {};
\node (v2000) at (13.25,0) [bluecirc] {};
\node (v2500) at (13.5,0.5) [bluecirc] {};
\node (v3000) at (13.5,0) [bluecirc] {};
\node (v4000) at (13.5,0.25) [bluecirc] {};
\node (v5000) at (13,0.25) [bluecirc]{};
\node (v5500) at (13,0.5) [bluecirc]{};

\node (v1010) at (13,0.75) [bluecirc]{};
\node (v1020) at (13.5,0.75) [bluecirc]{};

\node (v1001) at (13,1) [bluecirc]{};
\node (v1002) at (13.5,1) [bluecirc]{};

\node (v1003) at (13,-0.25) [bluecirc] {};
\node (v1004) at (13,-0.5) [bluecirc] {};
\node (v1005) at (13,-0.75) [bluecirc] {};
\node (v1006) at (13,-1) [bluecirc] {};

\node (v3003) at (13.5,-0.25) [bluecirc] {};
\node (v3004) at (13.5,-0.5) [bluecirc] {};
\node (v3005) at (13.5,-0.75) [bluecirc] {};
\node (v3006) at (13.5,-1) [bluecirc] {};

\draw [thick] (v1000) to (v2000);
\draw [thick] (v1000) to (v5000);
\draw [thick] (v2000) to (v3000);
\draw [thick] (v3000) to (v4000);
\draw [thick] (v5000) to (v5500);
\draw [thick] (v4000) to (v2500);

\draw [thick] (v5500) to (v1010);
\draw [thick] (v2500) to (v1020);

\draw [thick] (v1010) to (v1001);
\draw [thick] (v1020) to (v1002);

\draw [thick] (v1000) to (v1003);
\draw [thick] (v1003) to (v1004);
\draw [thick] (v1004) to (v1005);
\draw [thick] (v1005) to (v1006);

\draw [thick] (v3000) to (v3003);
\draw [thick] (v3003) to (v3004);
\draw [thick] (v3004) to (v3005);
\draw [thick] (v3005) to (v3006);

\node (l1) at (0.25,-0.5) {\small{$A_1$}};
\node (l2) at (2.25,-0.75) {\small{$A_2$}};
\node (l3) at (6.25,-1) {\small{$A_3$}};
\node (l4) at (13.25,-1.25) {\small{$A_4$}};


\draw[step=0.25cm,gray,very thin] (-0.2,-0.2) grid (13.9,1.49);
\draw[step=0.25cm,gray,very thin] (-0.2,-0.2) grid (13.9,-1.49);

\end{tikzpicture}
\caption{A counter example to uniform upper bound on returning probability}
\label{counterexample}
\end{figure}
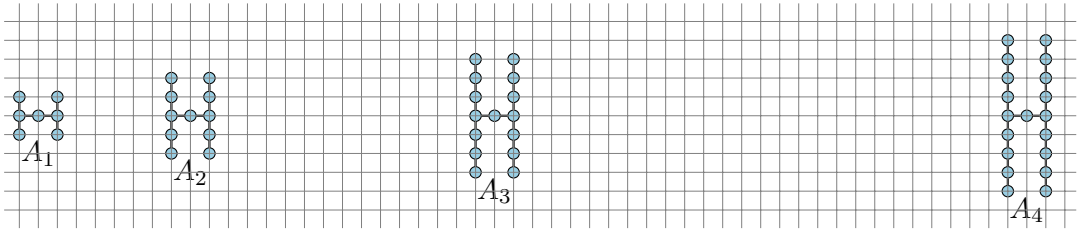 

Using Wiener's test, it is easy to see $A$ is a transient subset. However, for points $a_k=(2^k,0,0)\in A$, $k\ge 1$, in order to have a simple random walk starting at $a_k$ never returns to $A$, we must have the first $k$ steps of the random walk be along the $z-$coordinate. Thus
$$
P_{a_k}(T_{A}=\infty)<\frac{1}{3^k}, 
$$
which implies that 
$$
\lim_{k\to\infty} P_{a_k}(T_{A}<\infty)\ge \lim_{k\to\infty} \left(1-\frac{1}{3^k}\right)=1.  
$$
\begin{remark}
It would be interesting to characterize uniformly transient sets i.e. sets with uniformly bounded return probabilities. 
\end{remark}

Fortunately, for the specific transient subset $\mathcal{\tilde D}_\infty$, since it becomes more and more sparse as $x\to\infty$, we can still have: 

\begin{lemma}
\label{lemma uniform bound return}
For any $\ep>0$, there is a $c_{\ep,1}>0$ such that 
\beq
\label{returning probability uniform}
\sup_{k\ge 1}P_{x_k}(\bar T_{\mathcal{\tilde D}_\infty}<\infty)\le 1-c_{\ep,1}.
\eeq
\end{lemma}
\begin{proof}
With \eqref{returning probability<1} showing all returning probabilities are strictly less than 1, it is sufficient for us to show that
\beq
\label{returning probability uniform 1}
\limsup_{k\to\infty} P_{x_k}(\bar T_{\mathcal{\tilde D}_\infty}<\infty)<1. 
\eeq 
Actually, here we prove a stronger statement
\beq
\label{returning probability uniform 2}
\lim_{k\to\infty} P_{x_k}(\bar T_{\mathcal{\tilde D}_\infty}<\infty)=P_0(\bar T_0<\infty)<1. 
\eeq 
Note that for each $k$
\begin{align*}
&P_{x_k}(\bar T_{\mathcal{\tilde D}_\infty}<\infty)> P_{x_k}(\bar T_{x_k}<\infty)=P_0(\bar T_0<\infty),\\
&P_{x_k}(\bar T_{\mathcal{\tilde D}_\infty}<\infty)\le P_{x_k}(\bar T_{x_k}<\infty)+P_{x_k}(T_{\mathcal{\tilde D}_\infty\setminus \{x_k\}}<\infty),
\end{align*}
and that 
$$
P_{x_k}(\bar T_{\mathcal{\tilde D}_\infty\setminus \{x_k\}}<\infty)\le \sum_{i=1}^{k-1} P_{x_k}(T_{x_i}<\infty)+\sum_{i=k+1}^{\infty} P_{x_k}( T_{x_i}<\infty). 
$$
It suffices for us to show that 
\beq
\label{tail 1}
\lim_{k\to\infty} \sum_{i=1}^{k-1} P_{x_k}(T_{x_i}<\infty)=0,
\eeq
and that 
\beq
\label{tail 2}
\lim_{k\to\infty} \sum_{i=k+1}^{\infty} P_{x_k}(T_{x_i}<\infty)=0. 
\eeq
To show \eqref{tail 1} and \eqref{tail 2}, we first note the well known result that there is a $C<\infty$ such that for any $x\not=y\in \ZZ^3$, 
$$
P_x(T_y<\infty)\le \frac{C}{|x-y|}.
$$
First, to show \eqref{tail 1} we have according to \eqref{lower bound difference n_k}, for any $k\ge 8$
\beq
\label{tail 1 1}
\sum_{i=1}^{k-1} P_{x_k}(T_{x_i}<\infty)\le \sum_{i=1}^{k-1} \frac{C}{|x_k-x_i|}\le 2C\sum_{i=1}^{k-1} \frac{1}{\int_{i}^{k} \log^{1+\ep}(x)dx}.
\eeq
Thus it is again sufficient to show that 
\beq
\label{tail 1 2}
\lim_{k\to\infty} \sum_{i=1}^{k-1} \frac{1}{\int_{i}^{k} \log^{1+\ep}(x)dx}=0. 
\eeq
Note that 
\beq
\label{tail 1 3}
\sum_{i=1}^{k-1} \frac{1}{\int_{i}^{k} \log^{1+\ep}(x)dx}=\sum_{i=1}^{[k^{1/2}]} \frac{1}{\int_{i}^{k} \log^{1+\ep}(x)dx}+\sum_{i=\left\lceil k^{1/2}\right\rceil}^{k-1} \frac{1}{\int_{i}^{k} \log^{1+\ep}(x)dx}.
\eeq
For each $k\ge 8$ and $i\le [k^{1/2}]$, we have 
$$
\int_{i}^{k} \log^{1+\ep}(x)dx\ge \int_{k/2}^{k} \log^{1+\ep}(x)dx\ge \int_{k/2}^{k}1dx= k/2. 
$$
Thus 
\beq
\label{tail 1 4}
\sum_{i=1}^{[k^{1/2}]} \frac{1}{\int_{i}^{k} \log^{1+\ep}(x)dx}\le\sum_{i=1}^{[k^{1/2}]} \frac{2}{k}\le \frac{2}{k^{1/2}}=o(1). 
\eeq
Then for each $k\ge 8$ and $i\in \left[\left\lceil k^{1/2}\right\rceil, k-1\right]$, 
$$
\int_{i}^{k} \log^{1+\ep}(x)dx\ge \int_{i}^{k} \log^{1+\ep}(k^{1/2})dx= \frac{1}{2^{1+\ep}}(k-i) \log^{1+\ep}(k).
$$
Thus 
\beq
\label{tail 1 5}
\sum_{i=\left\lceil k^{1/2}\right\rceil}^{k-1} \frac{1}{\int_{i}^{k} \log^{1+\ep}(x)dx}\le \frac{2^{1+\ep}}{\log^{1+\ep}(k)}\sum_{i=1}^k \frac{1}{i}. 
\eeq
Noting that 
$$
\sum_{i=1}^k \frac{1}{i}\le 1+\int_1^k\frac{1}{x}dx=1+\log(k)
$$
one can immediately have 
\beq
\label{tail 1 6}
\sum_{i=\left\lceil k^{1/2}\right\rceil}^{k-1} \frac{1}{\int_{i}^{k} \log^{1+\ep}(x)dx}\le \frac{2^{1+\ep}}{\log^{1+\ep}(k)}\sum_{i=1}^k \frac{1}{i}\le \frac{2^{1+\ep} [1+\log(k)]}{\log^{1+\ep}(k)}=o(1). 
\eeq
Combining \eqref{tail 1 2}, \eqref{tail 1 4} and \eqref{tail 1 6}, we obtain \eqref{tail 1}. 

Then, to show \eqref{tail 2} we have according to \eqref{lower bound difference n_k}, for any $k\ge 8$
\beq
\label{tail 2 1}
\sum_{i=k+1}^{\infty} P_{x_k}(T_{x_i}<\infty)\le \sum_{i=k+1}^{\infty} \frac{C}{|x_i-x_k|}\le 2C\sum_{i=k+1}^{\infty} \frac{1}{\int_{k}^{i} \log^{1+\ep}(x)dx}
\eeq
Thus it is again sufficient to show that 
\beq
\label{tail 2 2}
\lim_{k\to\infty} \sum_{i=k+1}^{\infty} \frac{1}{\int_{k}^{i} \log^{1+\ep}(x)dx}=0. 
\eeq
Now for each $k$ we separate the infinite summation in \eqref{tail 2 2} as 
\beq
\label{tail 2 3}
\sum_{i=k+1}^{\infty} \frac{1}{\int_{k}^{i} \log^{1+\ep}(x)dx}=\sum_{i=k+1}^{k^2} \frac{1}{\int_{k}^{i} \log^{1+\ep}(x)dx}+\sum_{i=k^2+1}^{\infty} \frac{1}{\int_{k}^{i} \log^{1+\ep}(x)dx}. 
\eeq
For its first term we use similar calculation as in \eqref{tail 1 5} and have 
\beq
\label{tail 2 4}
\sum_{i=k+1}^{k^2} \frac{1}{\int_{k}^{i} \log^{1+\ep}(x)dx}\le  \frac{1}{\log^{1+\ep}(k)}\sum_{i=k+1}^{k^2}\frac{1}{i-k}\le  \frac{1}{\log^{1+\ep}(k)}\sum_{i=1}^{k^2}\frac{1}{i}.
\eeq
And since 
$$
\sum_{i=1}^{k^2} \frac{1}{i}\le 1+\int_1^{k^2}\frac{1}{x}dx=1+2\log(k)
$$
we have 
\beq
\label{tail 2 5}
\sum_{i=k+1}^{k^2} \frac{1}{\int_{k}^{i} \log^{1+\ep}(x)dx}\le  \frac{1+2\log(k)}{\log^{1+\ep}(k)}=o(1).
\eeq
At last for the second term in \eqref{tail 2 3}, we have for each $k\ge 8$ and $i\ge k^2+1$, 
$$
\int_{k}^{i} \log^{1+\ep}(x)dx\ge \int_{i^{1/2}}^{i} \log^{1+\ep}(x)dx\ge (i-i^{1/2})\log^{1+\ep}(i^{1/2})\ge \frac{1}{2^{2+\ep}} i \log^{1+\ep}(i).
$$
Thus
\beq
\label{tail 2 6}
\sum_{i=k^2+1}^{\infty} \frac{1}{\int_{k}^{i} \log^{1+\ep}(x)dx}\le 2^{2+\ep}\sum_{i=k^2+1}^{\infty} \frac{1}{i \log^{1+\ep}(i)}.
\eeq
Finally, noting that 
$$
\sum_{i=3}^{\infty} \frac{1}{i \log^{1+\ep}(i)}\le \int_2^\infty \frac{1}{x \log^{1+\ep}(x)}dx=\frac{1}{\ep \log^{\ep}(2)}<\infty,
$$
we have the tail term 
\beq
\label{tail 2 7}
\sum_{i=k^2+1}^{\infty} \frac{1}{i \log^{1+\ep}(i)}=o(1)
\eeq
as $k\to\infty$. Thus combining \eqref{tail 2 2}- \eqref{tail 2 7}, we have shown \eqref{tail 2} and thus finished the proof of this lemma. 
\end{proof}

\begin{proof}[Proof of Theorem \ref{Theorem 3}]
With Lemma \ref{lemma uniform bound return}, and recalling that 
$$
\mathcal{\tilde D}_N=\mathcal{\tilde D}_\infty\cap \mathcal{D}_{N}
$$
and
$$
C_N=\left|\mathcal{\tilde D}_N\right|=\sup\{k: \  n_k\le N\},
$$
we can define the stopping times $\bar T_{\mathcal{\tilde D}_{[N/3]},0}=0$, 
$$
\bar T_{\mathcal{\tilde D}_{[N/3]},1}=\inf\left\{n>0,  \ X_{3,n}\in \mathcal{\tilde D}_{[N/3]}\right\}
$$
and for all $k\ge 2$
$$
\bar T_{\mathcal{\tilde D}_{[N/3]},k}=\inf\left\{n>\bar T_{\mathcal{\tilde D}_{[N/3]},k-1}, \  X_{3,n}\in \mathcal{\tilde D}_{[N/3]}\right\}.
$$
Then by Lemma \ref{lemma uniform bound return}, one can immediately see that for any $k\ge 0$ 
$$
P\left(T_{\mathcal{\tilde D}_{[N/3]},k+1}<\infty\Big| \bar T_{\mathcal{\tilde D}_{[N/3]},k}<\infty \right)\le P_{X_{3,\bar T_{\mathcal{\tilde D}_{[N/3]},k}}}(\bar T_{\mathcal{\tilde D}_\infty}<\infty)\le 1-c_{\ep,1},
$$
and thus 
\beq
\label{equation theorem 3 1}
\begin{aligned}
P\left(T_{\mathcal{\tilde D}_{[N/3]},C_{[N/3]}}<\infty\right)&=\prod_{k=0}^{C_{[N/3]}-1}P\left(T_{\mathcal{\tilde D}_{[N/3]},k+1}<\infty\Big| \bar T_{\mathcal{\tilde D}_{[N/3]},k}<\infty \right)\\
&\le(1-c_{\ep,1})^{C_{[N/3]}}.
\end{aligned}
\eeq

By Lemma \ref{lemma C_N} we have 
\beq
\label{equation theorem 3 2}
C_{[N/3]}\ge 2^{-\ep-1} [N/3]\log^{-1-\ep}([N/3]) \ge \frac{2^{-\ep-2}}{3} N\log^{-1-\ep}(N)
\eeq
for all $N\ge 4$. Thus combining \eqref{equation theorem 3 1} and \eqref{equation theorem 3 2}
\beq
\begin{aligned}
P\left(\overset{\nearrow}{\mathcal{P}}\subseteq {\rm Trace}\big(\{X_{3,n}\}_{n=0}^\infty\big)\right)&\le P\left(\mathcal{D}_{[N/3]}\subseteq {\rm Trace}\big(\{X_{3,n}\}_{n=0}^\infty\big)\right)\\
&\le P\left(\mathcal{\tilde D}_{[N/3]}\subseteq {\rm Trace}\big(\{X_{3,n}\}_{n=0}^\infty\big)\right)\\
&\le P\left(T_{\mathcal{\tilde D}_{[N/3]},C_{[N/3]}}<\infty\right)\\
&\le \exp\left(-c_\ep N\log^{-1-\ep}(N)\right)
\end{aligned}
\eeq
where $c_\ep=-\frac{2^{-\ep-2}}{3} \log(1-c_{\ep,1})$. And the proof of Theorem \ref{Theorem 3} is complete. \end{proof}

\section{Discussions}
\label{Section discussion}

In Conjecture \ref{conjecture 3}, we conjecture that the cover probability should have exponential decay just as the $d\ge 4$ case. This conjecture seems to be supported by the following preliminary simulation which shows the log-plot of probabilities that the first $5000$ steps of a 3 dimensional simple random walk starting at 0 cover $\mathcal{D}_i=\{(0,0,0), (1,1,1),\cdots, (i,i,i)\}$ for $i=1,2,\cdots,9$. 

\begin{figure}[h]
\label{fig: 3}
\center
\includegraphics[width=0.65\columnwidth]{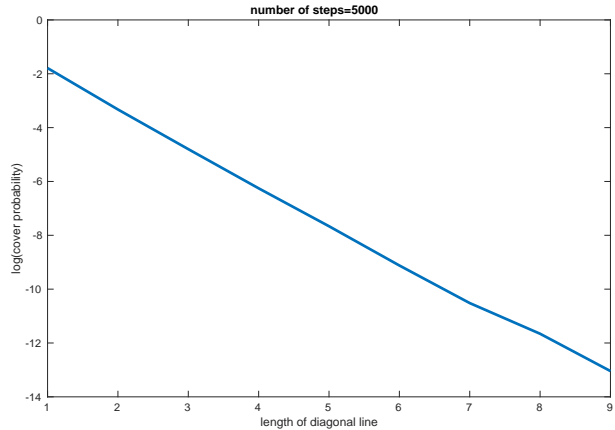} 
\caption{log-plot of covering probabilities of $\mathcal{D}_i$, $i=1,2,\cdots,9$}
\end{figure}

The simulation result above seems to indicate that after taking logarithm, the covering probability decays almost exactly as a linear function, which implies the exponential decay we predicted, indicating that the upper bound we found in Theorem \ref{Theorem 3} is not sharp. For $N=9$, if Theorem \ref{Theorem 3} were sharp and there were a correction greater than $\log(N)$ in the exact decaying rate, then in the log-plot, it would cause the point to be $\log\log(9)\ge 0.787$ above the line. This is not seen in the simulation. However, the simulation above does not rule out the possibility that there is correction of a smaller order than $\log(N)$, since it could be so small for the initial 9 $i$'s and thus has not be seen significantly yet in the current simulation.

Another possible approach towards a sharp asymptotic is noting that although $\{\hat X_{2,n}\}_{n=0}^\infty$ is recurrent and will return to 0 with probability 1, the expected time between each two successive returns is $\infty$. Moreover, in order to cover $\overset{\nearrow}{\mathcal{P}}$, only those returns to diagonal before that $\{X_{3,n}\}_{n=0}^\infty$ has left $B_2(0,N)\supset B_1(0,N)$ forever could possibly help. This observation, together with the tail probability asymptotic estimations using local central limit theorem and techniques in \cite{excursion2} and \cite{excursion22} applied on the non simple random walk $\{\hat X_{2,n}\}_{n=0}^\infty$, and some large deviation argument, enable us to find a proper value of $T$ such that 
\begin{itemize}
\item with high probability $\{X_{3,n}\}_{n=T}^\infty\cap B_2(0,N)=\O$,
\item with high probability $\{\hat X_{2,n}\}_{n=0}^T$ will not return to 0 for $[N/3]$ times or more. 
\end{itemize} 
Right now this approach can only give us the following weaker upper bound (a detailed proof can be found in technical report \cite{weak_result}):
\begin{proposition}
\label{proposition weak}
There are $c, C\in (0,\infty)$ such that for any nearest neighbor path $\mathcal{P}=(P_0,P_1,\cdots, P_K)\subset \ZZ^3$ connecting 0 and $\partial B_1(0,N)$,
$$
P\left({\rm Trace}(\mathcal{P})\subseteq {\rm Trace}\big(\{X_{3,n}\}_{n=0}^\infty\big) \right)\le C\exp\left(-c N^{1/3}\right). 
$$
\end{proposition} 

However, this seemingly worse approach might have the potential to fully use the geometric property of path $\overset{\nearrow}{\mathcal{P}}$ to minimize the covering probability. Note that in order to cover $D_{[N/3]}$ we not only need $\{\hat X_{2,n}\}_{n=0}^\infty$ return to 0 for at least $[N/3]$ times before $\{X_{3,n}\}_{n=0}^\infty$ forever leaving $B_2(0,N)$, but also must have the locations of $X_{3,n}$ at such visits cover each point on the diagonal (let alone the request of covering the off diagonal points as well). I.e., define the stopping times $\tau_{l_3,0}=0$
$$
\tau_{l_3,1}=\inf\{n\ge 1: \ \hat X_{2,n}=0\}
$$
and for all $i\ge 2$
$$
\tau_{l_3,i}=\inf\{n> \tau_{l_3,i-1}: \ \hat X_{2,n}=0\}.
$$
Define 
$$
\{Z_{3,n}\}_{n=0}^\infty=\left\{X_{3,\tau_{l_3,n}}^1+X_{3,\tau_{l_3,n}}^2+X_{3,\tau_{l_3,n}}^{3}\right\}_{n=0}^\infty. 
$$
Noting that $\tau_{l_3,i}<\infty$ for any $i$, and that $\{X_{3,n}\}_{n=0}^\infty$ is translation invariant, $\{Z_{3,n}\}_{n=0}^\infty$ is a well defined one dimensional random walk with infinite range. And we have
$$
P\left({\rm Trace}(\mathcal{P})\subseteq {\rm Trace}\big(\{X_{3,n}\}_{n=0}^\infty\big) \right)\le P\left((0,1,\cdots, [N/3])\subseteq {\rm Trace}\big(\{Z_{3,n}\}_{n=0}^\infty\big)\right). 
$$
Thus Conjecture \ref{conjecture 3} would follow from the techniques described above for Proposition \ref{proposition weak} if the following conjecture is proved.
\begin{conjecture}
There is a $c \in (0,\infty)$ such that for any $N\ge 2$ 
$$
P\left((0,1,\cdots, [N/3])\subseteq {\rm Trace}\big(\{Z_{3,n}\}_{n=0}^{N^3}\big)\right)\le \exp(-c N).
$$
\end{conjecture}

\bibliography{mono}
\bibliographystyle{plain}

\end{document}